\documentclass[letterpaper,11pt,reqno]{amsart}

\usepackage{amsmath, amssymb, amscd, MnSymbol,mathrsfs}
\usepackage{verbatim}
\usepackage[utf8]{inputenc}

\usepackage{graphicx}

\newtheorem{theorem}{Theorem}[section] 
\newtheorem{prop}[theorem]{Proposition}
\newtheorem{cor}[theorem]{Corollary}

\newtheorem*{theorem*}{Theorem}

\theoremstyle{definition}
\newtheorem{definition}{Definition}

\newtheorem{example}{Example}

\newcommand{\var}{\mathcal{V}}
\newcommand{\qvar}{\mathcal{Q}}

\newcommand{\bydef}{:=}

\newcommand{\WSFi}{\mathsf{WS5}}

\newcommand{\CSub}{\mathbb{S}}
\newcommand{\CHom}{\mathbb{H}}
\newcommand{\CProd}{\mathbb{P}}
\newcommand{\CUProd}{\mathbb{P}_\mathsf{u}}

\newcommand{\KM}{\mathsf{KM}}

\newcommand{\impl}{\rightarrow}
\newcommand{\dimpl}{\leftarrow}

\newcommand{\Var}[1]{\mathcal{#1}}

\newcommand{\one}{\mathbf{1}}
\newcommand{\zero}{\mathbf{0}}
\newcommand{\mapping}{\longrightarrow}

\newcommand{\Con}{\mathsf{Con}}

\newcommand{\pfltr}[1]{[#1)}

\newcommand{\Heyt}{\BV{\mathsf{HA}}}
\newcommand{\HDP}{\BV{\mathsf{HDP}}}
\newcommand{\DHt}{\BV{\mathsf{DH}}}
\newcommand{\VMHA}{\BV{\mathsf{MH}}}
\newcommand{\VHRI}{\BV{\mathsf{HRI}}}
\newcommand{\VWSFi}{\BV{\mathsf{WS}5}}
\newcommand{\VSFi}{\BV{\mathsf{S}5}}

\newcommand{\MA}{\Alg{M}}
\newcommand{\two}{\Alg{2}}

\newcommand{\BRed}[1]{#1^\mathsf{b}}

\newcommand{\set}[2]{\{ #1 \ | \ #2 \}}

\newcommand{\pair}[2]{\langle #1, #2 \rangle}
\newcommand{\alg}[1]{\mathsf{#1}}

\newcommand{\Alg}[1]{\mathbf{#1}}

\newcommand{\Falgv}[1]{\Alg{F}_{\Var{V}}(#1)}
\newcommand{\Falg}[2]{\Alg{F}_{\Var{#2}}(#1)}

\newcommand{\falg}[2]{\Alg{F}_{#2}(#1)}

\newcommand{\pclass}[2]{#1/{#2}}

\newcommand{\BV}[1]{\boldsymbol{#1}}

\newcommand{\mvar}{\var_{min}}
\newcommand{\varq}{\qvar_\var}
\newcommand{\varmh}{\var_{mh}}

\newcommand{\Def}[1]{\textbf{\textit{#1}}}

\newcommand{\dneg}{\backneg}

\title[Projective Algebras in Varieties with Factor Congruences]{Projective Algebras and Primitive Subquasivarieties in Varieties with Factor Congruences}

\author{Alex Citkin}

\address{Metropolitan Telecommunications, USA} 

\email{acitkin@gmail.com}

\begin{document}

\begin{abstract} We prove that in the varieties where every compact congruence is a factor congruence and every nontrivial algebra contains a minimal subalgebra, a finitely presented algebra is projective if and only if it has every minimal algebra as its homomorphic image. Using this criterion of projectivity, we describe the primitive subquasivarieties of discriminator varieties that have a finite minimal algebra which embeds in every nontrivial algebra from this variety. In particular,	 we describe the primitive quasivarieties of discriminator varieties of monadic Heyting algebras, Heyting algebras with regular involution, Heyting algebras with dual pseudocomplement, and double-Heyting algebras.
\end{abstract} 


\maketitle




\section{Introduction}

It was observed by Ghilardi (see e.g. \cite{Ghilardi_Unification_1997}), that unification in algebraizable logics is directly related to projectivity of finitely presented algebras of the corresponding variety. This justifies the interest to studying of the projective finitely presented algebras in the varieties raised from logic. Finite projective Heyting algebras have rather a simple structure which had been described a long time ago (see \cite{Balbes_Horn_1970}). Projective finitely presented Heyting algebras have been described by Ghilardi who observed that any finitely generated subalgebra of a finitely presented projective Heyting
algebra is finitely presented as well as projective  (see e.g. \cite[p 109]{Ghilardi_Zawadowski_Book_2002}). Projectivity of Heyting and monadic Heyting algebras was studied by Grigolia (see e.g. \cite{Grigolia_Free_1995}). Using the duality theory he gave another criterion for a finitely generated Heyting algebra to be projective (see \cite[Theorem 3.1]{Grigolia_Free_1995}).     

The varieties corresponding to logics often contain a nontrivial algebra that embeds in every nontrivial algebra from this variety. For instance, some of such algebras are two-element Boolean algebra in the varieties of Heyting algebras, two-element Boolean algebra with the trivial interior operator in the varieties of interior algebras, etc. Also, these varieties often have congruences with rather nice properties, like equationally definable principal congruences (EDPC), etc. (see e.g. \cite{Blk_Pgz_3}). If such a variety is discriminator, Quackenbush proved the following criterion for projectivity of finite algebras \cite[Theorem 5.2]{Quackenbush_Demi-semi-primal_1971}: in a variety generated by a quasi-primal algebra, a finite algebra is projective if and only if it has every minimal algebra as a homomorphic image. In Section \ref{sec-proj} we extend this criterion to finitely-presented algebras from congruence distributive varieties with factor congruences in which every nontrivial algebra contains a minimal subalgebra.

Then, in Section \ref{sec-prim}, we employ the criterion established in Section \ref{sec-proj} and we study the primitive subquasivarieties of the congruence distributive varieties with factor congruences, in which every nontrivial algebra contains a minimal subalgebra. The primitive quasivarieties are algebraic counterparts of the hereditarily structurally complete consequence relations (see e.g. \cite{Olson_Raftery_Alten_2008,Cintula_Metcalfe_2009}). 

In Section \ref{sec-appl}, we apply the obtained results to the following varieties (see the definitions in Section \ref{sec-prel}):
\begin{itemize}
	\item[(a)] $\WSFi$-algebras (see \cite{Bezhanishvili_G_Glivenko_2001});
	\item[(b)] Heyting algebras with dual pesudocomplement (see e.g. \cite{Taylor_C_Discriminator_2016}); 
	\item[(c)] Double Heyting algebras (see e.g. \cite{Taylor_C_Discriminator_2016});
	\item[(b)] Heyting algebras with involution (see e.g. \cite{Meskhi_Discriminator_1982} or symmetrical Heyting algebras (see e.g. \cite{Iturrioz_Modal_1982}).
\end{itemize}
We also look at monadic Heyting algebras (see \cite{G_Bez_MIPC_1}) and $\KM$-algebras (see e.g. \cite{Muravitsky_Logic_2014}) as well as Heyting algebras with successor (see \cite{Caicedo_Cignoli_Algebraic_2001}) and frontal algebras (see \cite{Castiglioni_et_Frontal_2010}).

\section{Preliminaries} \label{sec-prel}

\subsection{Algebras and Congruences}

We consider algebras of arbitrary but fixed finite similarity type. Given algebra $\Alg{A}$, $\Con \Alg{A}$ denotes the congruence lattice of algebra $\Alg{A}$, and $\varepsilon_\Alg{A}$ and $\tau_\Alg{A}$ denote respectively the smallest and the greatest elements of $\Con\Alg{A}$, that is, identity and trivial congruences. We assume that the reader is familiar with the basic notions of universal algebra from \cite{Burris_Sanka} or \cite{GraetzerB}. 

Recall that a nontrivial algebra $\Alg{A}$ is \Def{simple}, if it has exactly two congruence, or, in other words, when $\Alg{A}$ has exactly two homomorphic images: a trivial algebra and itself.

Let $\Var{K}$ be a class of algebras and $\Alg{A}$ be an algebra. Then a congruence $\theta$ on $\Alg{A}$ is said to be a $\Var{K}$-congruence if $\Alg{A}/\theta \in \Var{K}$. The set of all $\Var{K}$-congruences on $\Alg{A}$ is denoted by $\Con_\Var{K} \Alg{A}$. If class $\Var{K}$ is closed under formation of subdirect products, then $\Con_\Var{K} \Alg{A}$ forms a complete lattice (see e.g. \cite{GorbunovBookE}). Hence, for an algebra $\Alg{A}$ from a class $\Var{K}$ closed under subdirect products, for any elements $\alg{a}_i,\alg{b}_i \in \Alg{A}, i < n$ there is a smallest $\Var{K}$-congruence $\theta_\Var{K}$ such that $\alg{a}_i \equiv \alg{b}_i \pmod {\theta_\Var{K}}$ for all $i < n$, and we say that the pairs $(\alg{a}_i, \alg{b}_i), i<n$ generate $\theta_\Var{K}$. The $\Var{K}$-congruences generated by a finite set of pairs of elements are called \Def{compact}, and if a $\Var{K}$-congruence is generated  by a single pair $(\alg{a},\alg{b})$, it is called \Def{principal} and is denoted by $\theta_\Var{K}(\alg{a},\alg{b})$ (and we omit reference to $\Var{K}$ when no confusion arises). If $\theta$ is a congruence, by $\pclass{\alg{a}}{\theta}$ we denote a $\theta$-congruence class containing element $\alg{a}$.

As usual, if $\Var{K}$ is a class of algebras, $\CSub\Var{K},\CHom\Var{K}, \CProd\Var{K},\CUProd\Var{K}$ denote respectively the classes of isomorphic copies of all subalgebras , homomorphic images, direct products and ultraproducts of algebras from $\Var{K}$.
Recall (see e.g. \cite{GorbunovBookE}) that a class of algebras $\Var{K}$ is a \Def{variety} if it is closed under $\CHom, \CSub,\CProd$, class $\Var{K}$ is a \Def{quasivariety} if $\Var{K}$ is closed under $\CSub,\CProd,\CUProd$, and $\Var{K}$ is a \Def{prevariety} if $\Var{K}$ is closed under $\CSub,\CProd$. A class of algebras $\Var{K}$ is said to be axiomatizable if for some set of first-order sentences $\Gamma$, we have $\Alg{A} \in \Var{K}$ if and only if every sentence from $\Gamma$ is valid in $\Alg{A}$. The following proposition is a consequence of the Maltsev-Vaught theorem.

\begin{prop} \label{pr-prev-quas} (see e.g. {\cite[Corollary 2.3.3]{GorbunovBookE}}) For any axiomatizable class of algebras $\Var{K}$, the prevariety $\CSub\CProd\Var{K}$ generated by $\Var{K}$ is a quasivariety. In particular, a prevariety $\Var{P}$ is a quasivariety if and only if  $\Var{P}$ is an axiomatizable class of algebras. 
\end{prop} 

\subsubsection{Factor Congruences} The notion of factor congruence plays an important role in what follows.
\begin{definition} \cite[Definition 7.4]{Burris_Sanka}
	Let $\Alg{A}$ be an algebra and $\theta$ be a congruence on $\Alg{A}$. Then $\theta$ is a \Def{factor congruence} if there is a congruence $\theta' \in \Con\Alg{A}$ such that $\theta \cap \theta' = \varepsilon$ and $\theta \lor \theta' = \tau$ and $\theta$ permutes with $\theta'$. The pair $\theta,\theta'$ is called a \Def{pair of factor congruences} and $\theta'$ is a \Def{complement} of $\theta$.
\end{definition}
For any pair of factor congruences the following holds.

\begin{prop} \cite[Theorem 7.5]{Burris_Sanka} \label{pr-factor}
	If $\theta,\theta'$ is a pair of factor congruences on $\Alg{A}$, then
	\[
	\Alg{A} \cong \Alg{A}/\theta \times \Alg{A}/\theta'.
	\]
\end{prop} 

\subsection{Discriminator Varieties} Let us recall a definition of discriminator varieties: an important example of varieties in which for every algebra each compact congruence is a factor congruence.  

A variety $\var$ is said to be \Def{discriminator} (see e.g. \cite{Burris_Sanka}) if there is a term $t(x,w,z)$ such that for every nontrivial algebra $\Alg{A} \in \var$ and any elements $\alg{a},\alg{b},\alg{c} \in \Alg{A}$,
\[
t(\alg{a},\alg{b},\alg{c}) = \begin{cases} \alg{a} \quad \text{if } \alg{a} \neq \alg{b} \\
\alg{c} \quad \text{if } \alg{a} = \alg{b}.
\end{cases}
\]

Discriminator varieties have very nice properties. We will use the following (see \cite[Theorem 1.1]{Andreka_Jonsson_Nemeti_1991}).

\begin{prop} \label{pr-discr} Let $\var$ be a discriminator variety. Then the following holds for every algebra $\Alg{A} \in \var$.
	\begin{itemize}
		\item[(a)] $\var$ is semisimple, that is every subdirectly irreducible algebra is simple;
		\item[(b)] If algebra $\Alg{A}$ is finite, then $\Alg{A}$ is a direct product of simple algebras;
		\item[(c)] Every compact congruence on $\Alg{A}$ is principal;
		\item[(d)] Every principal congruence on $\Alg{A}$ is a factor congruence.
	\end{itemize}
\end{prop}

\subsection{Finitely Presented Algebras}
In this section we recall the definition of finitely presented algebras.

A pair $\pair{X}{\Delta}$, where $X$ is a set of variables and $\Delta$ is a set of atomic formulas, all variables of which are in $X$, is called a \Def{defining pair}.

\begin{definition} \label{def-finpres}
	Let $\Var{K}$ be a prevariety. A pair  $\pair{X}{\Delta}$ \Def{defines} and algebra $\Alg{A} \in \Var{K}$ if there is a map $\varphi: X \mapping \Alg{A}$ such that
	\begin{itemize}
		\item[(a)] the set $\varphi(X)$ generates $\Alg{A}$ and every formula from $\Delta$ holds in $\Alg{A}$ under valuation $\varphi$;
		\item[(b)] For every algebra $\Alg{B} \in \Var{K}$ and map $\psi: X \mapping \Alg{B}$, if every formula from $\Delta$ holds under valuation $\psi$, then $\psi$ can be extended to a homomorphism of $\Alg{A}$ to $\Alg{B}$,
	\end{itemize}
	and $\varphi$ is called a \Def{defining valuation}.
\end{definition}

An algebra defined by a pair $\pair{X}{\Delta}$ is unique (up to an isomorphism) and it is denoted by $\Falg{X,\Delta}{K}$. Let us recall (see \cite[Theorem 2.1.3]{GorbunovBookE}) that in every prevariety $\Var{K}$, algebra $\Falg{X,\Delta}{K}$ exists for any defining pair $\pair{X}{\Delta}$.  

\begin{prop} \label{pr_fp}
	Let $\pair{X}{\Delta}$ be a defining pair and $\Var{K}_0$ and $\Var{K}_1$ be prevarieties such that $\Var{K}_0 \subseteq \Var{K}_1$. Then
	\[
	\begin{array}{ll}
	(a) & \falg{X,\Delta}{K_0} \text{ is a homomorphic image of } \falg{X,\Delta}{K_1};\\
	
	(b) & 	\text{if } \falg{X,\Delta}{K_1} \in \Var{K}_0 \text{, then } \falg{X,\Delta}{K_0} \cong \falg{X,\Delta}{K_1}.
	\end{array}
	\]
\end{prop}
\begin{proof} Let $\pair{X}{\Delta}$ be a defining pair and $\Var{K}_0$ and $\Var{K}_1$ be some prevarieties such that $\Var{K}_0 \subseteq \Var{K}_1$. Suppose $\varphi_i, i=0,1$ are defining valuations. Let us consider the map $\psi: \falg{X,\Delta}{K_1} \mapping \falg{X,\Delta}{K_0}$ defined in the following way:
	\[
	\psi: \varphi_1(x) \mapsto \varphi_0(x) 
	\]
	By Definition \ref{def-finpres}(a), all formulas from $\Delta$ are valid under valuation $\varphi_0$. Hence, because $\falg{X,\Delta}{K_0} \in \Var{K}_0 \subseteq \Var{K}_1$, by Definition \ref{def-finpres}(b), the map $\psi$ can be extended to a homomorphism which is a homomorphism of $\falg{X,\Delta}{K_1}$  onto $\falg{X,\Delta}{K_0}$, for elements $\varphi_0(x), x \in X$ generate algebra $\falg{X,\Delta}{K_0}$. 
	
	Suppose that $\falg{X,\Delta}{K_1} \in \Var{K}_0$. It is clear that, because $\Var{K}_0 \subseteq \Var{K}_1$, if conditions of Definition \ref{def-finpres} hold in $\Var{K}_1$, they hold in $\Var{K}_0$. Thus, $\falg{X,\Delta}{K_1}$ is an algebra defined in $\Var{K}_0$ by pair $\pair{X}{\Delta}$, and hence, $\falg{X,\Delta}{K_1}$ is isomorphic to $\falg{X,\Delta}{K_0}$, for in a given prevariety, a defining pair defines an algebra uniquely up to an isomorphism.
\end{proof}	

An algebra $\Alg{A}$ is said to be \Def{finitely $\Var{K}$-presented} if $\Alg{A}$ can be define in $\Var{K}$ by a defining pair $\pair{X}{\Delta}$ in which $X$ and $\Delta$ are finite sets.

Let us observe that an algebra $\Alg{A}$ is finitely $\Var{K}$-presented if and only if there is a compact $\Var{K}$-congruence $\theta_\Var{K}(\Alg{A})$ on $\falg{n}{\Var{K}}$ such that
\[
\Alg{A} \cong \falg{n}{\Var{K}}/\theta_\Var{K}(\Alg{A}),
\]
where $\falg{n}{\Var{K}}$ is a free algebra of a finite rank $n$.

Finitely $\Var{K}$-presented systems play an important role in the theory of quasivarieties. In particular, the following holds.

\begin{prop} \label{pr-qfpgen} \cite[Proposition 2.1.18]{GorbunovBookE} In a prevariety $\Var{K}$ every algebra is a direct limit of finitely $\Var{K}$-presented algebras from $\Var{K}$. Every quasivariety $\Var{Q}$ is generated by its finitely $\Var{Q}$-presented algebras.
\end{prop}

\subsection{Classes of Algebras}

In this section we recall information about classes of algebras that we consider in the squeal.

\subsubsection{Heyting Algebras} In this section we recall the definition and some properties of Heyting algebras that we use in the squeal. All necessary details can be found in \cite{Rasiowa_Sikorski} (where Heyting algebras are called pseudo-Boolean algebras.)

\Def{Heyting algebra} is an algebra $\Alg{A} = (\alg{A};\land,\lor,\impl,\zero,\one)$, where $\Alg{A} = (\alg{A};\land,\lor,\zero,\one)$ is a bounded distributive lattice and $\impl$ is a relative pseudocomplement, that is for each $\alg{a},\alg{b},\alg{c} \in \alg{A}$ we have \[
\alg{a} \land \alg{b} \leq \alg{c} \text{ if and only if } \alg{c} \leq \alg{a} \impl \alg{b}.
\] 
As usual, $\alg{a} \impl \zero$ is abbreviated to $\neg \alg{a}$, and $\neg$ is a pseudocomplement. 

The set of all Heyting algebras forms a variety denoted by $\Heyt$.

\begin{definition} \label{def-filter}
	If $\Alg{A}$ is a Heyting algebra a subset $\alg{F} \subseteq \Alg{A}$ is a \Def{Heyting filter} on $\Alg{A}$ (h-filter for short) if the following holds: (a) $\one \in \alg{F}$; (b) for every $\alg{a},\alg{b} \in \Alg{A}$, if $\alg{a} \leq \alg{b}$ and $\alg{a} \in \alg{F}$, then $\alg{b} \in \alg{F}$; (c)  $\alg{a} \land \alg{b} \in \alg{F}$ for every $\alg{a},\alg{b} \in \alg{F}$.
\end{definition}

It is clear that the meet of an arbitrary family of h-filters is an h-filter. If $\alg{B} \subseteq \Alg{A}$, by $\pfltr{\alg{B}}$ we denote the smallest h-filter containing $\alg{B}$, and we say that h-filter $\pfltr{\alg{B}}$ \Def{is generated by} $\alg{B}$. If $\alg{B} = \{\alg{b}\}$, we write $\pfltr{\alg{b}}$ omitting curly brackets.

There is a well-known one-to-one correspondence between h-filters and congruences: if $\theta$ is a congruence on a Heyting algebra $\Alg{A}$, then the set $\alg{F}_\theta \bydef \set{\alg{a} \in \Alg{A}}{\alg{a} \equiv \one \pmod{\theta}}$ forms an h-filter. And every congruence $\theta$ on a Heyting algebra $\Alg{A}$ is completely defined by congruence class $\alg{F}_\theta$.

Let $\Alg{A}$ be a Heyting algebra. An element $\alg{a} \in \Alg{A}$ is said to be \Def{regular} if $\neg\neg\alg{a} =  \alg{a}$ and $\alg{a}$ is said to be \Def{dense} if $\neg\alg{a} = \zero$. It is not hard to see that a Heyting algebra is Boolean if and only if $\one$ is the only its dense element, and the following holds.

\begin{prop} \label{pr-dense}
	Let $\Alg{A}$ be a Heyting algebra and $\theta$ be a congruence on $\Alg{A}$. Then $\Alg{A}/\theta$ is a Boolean algebra if and only if every dense element of $\Alg{A}$ is in $\alg{F}_\theta$.
\end{prop}

\subsubsection{Monadic Heyting Algebras}\footnote{All necessary details regarding monadic Heyting algebras the reader can find in \cite{G_Bez_MIPC_1}.}
A \Def{Monadic Heyting Algebra} is an algebra $\Alg{A} = (\alg{A};\land,\lor,\impl,\zero,\one,\Box,\Diamond)$, where $\Alg{A} = (\alg{A};\land,\lor,\impl,\zero,\one)$ is a Heyting algebra and $\Box$ and $\Diamond$ satisfy the following properties:
for each $\alg{a},\alg{b} \in \alg{A}$ we have 
\[
\begin{array}{lll}
\Box \alg{a} \leq \alg{a} & \quad &\alg{a} \leq \Diamond\alg{a};\\
\Box(\alg{a} \land \alg{b}) = \Box\alg{a} \land \Box\alg{b} &  \quad & \Diamond\alg{a} \lor \Diamond\alg{b} = \Diamond(\alg{a} \lor \alg{b});\\
\Box\one = \one &  \quad & \Diamond\zero = \zero;\\
\Box(\Box\alg{a} \lor \alg{b}) = \Box\alg{a} \lor\Box\Alg{b} & \quad & \Diamond(\Diamond\alg{a} \land \alg{b}) = \Diamond\alg{a} \land \Diamond\alg{b}. 
\end{array}
\] 

A set of all monadic Heyting algebras forms a variety denoted by $\VMHA$ (see e.g. \cite{G_Bez_MIPC_1}). 

Let $\two$ be a two-element monadic Heyting algebra. Obviously $\two$ is a subalgebra of any nontrivial $\VMHA$-algebra.

\subsubsection{$\VWSFi$-Algebras}

$\VWSFi$-algebra is a monadic Heyting algebra in which open elements, that is the elements $\alg{a}$ such that $\Box\alg{a} = \alg{a}$, form a Boolean subalgebra. $\VWSFi$-algebras admit identity $\Diamond x \cong \neg\Box\neg x$, thus, we can view $\VWSFi$-algebras as algebras of type $(\alg{A};\land\lor,\to,\zero,\one,\Box)$. An algebra $(\alg{A};\land\lor,\to,\zero,\one,\Box)$ is a $\VWSFi$-algebra if $(\alg{A};\land\lor,\to,\zero,\one)$ is a Heyting algebra (an \Def{h-reduct} of $\Alg{A}$) and the following holds: for every $\alg{a},\alg{b} \in \alg{A}$
\begin{equation}
	\begin{array}{ll}
	(a) & \Box\one = \one;\\
	(b) & \Box\alg{a} = \alg{a};\\
	(c) & \Box(\alg{a} \land \alg{b}) = \Box\alg{a} \land \Box\alg{b};\\
	(d) & \Box\Box\alg{a} = \alg{a};\\
	(e) & \Box(\alg{a} \lor \Box\alg{b}) = \Box\alg{a} \lor \Box\alg{b}.	
	\end{array} \label{eq-ws5} \tag{WS5}
\end{equation}

Similarly to monadic Heyting algebras, $\two$ is a subalgebra of any nontrivial $\VWSFi$-algebra.

\begin{prop}[see e.g. {\cite{G_Bez_MIPC_1}}] \label{pr-WS5} 
	Let $\Alg{A}$ be a $\VWSFi$-algebra. Then
	\begin{itemize}
		\item[(a)] $\Alg{A}$ is subdirectly irreducible if and only if it is simple;
		\item[(b)] $\Alg{A}$ is simple if and only if it has exactly two open elements: $\zero$ and $\one$;
		\item[(c)] If $\Alg{A}$ is finite, then $\Alg{A}$ is a direct product of simple $\VWSFi$-algebras.
	\end{itemize}
\end{prop}

Also, the following holds.

\begin{prop} \label{pr-WS5fin} Let $\Alg{A}$ be a finitely generated $\VWSFi$-algebra. If h-reduct of $\Alg{A}$ is Boolean, then $\Alg{A}$ is finite. 
\end{prop}	
\begin{proof}
	If h-reduct of the algebra is Boolean, then $\Alg{A}$ is an $\VSFi$-algebra, and every finitely-generated $\VSFi$-algebra is finite (see e.g. \cite[Corollary 5.19]{Chagrov_Zakh}).
\end{proof}

\subsubsection{$\VWSFi$-algebras with Compatible Operations}

If we extend the signature of $\VWSFi$-algebras by new operations $f_i, i<n$, these operations are said to be \Def{compatible} if every $\VWSFi$-congruence is a congruence on the algebra in extended signature (see \cite{Blk_Pgz_3}). In other terms, new operations are compatible if $\VWSFi$-filters coincide with filters of extended algebras. The algebras that we consider below can be viewed as $\VWSFi$-algebras with compatible operations. Let us observe the following.

The filters generated by a set of elements of a given $\VWSFi$-algebra with compatible operations can be described in the following way. 

\begin{prop} \label{pr-genfltr}
	Let $\Alg{A}$ be a $\VWSFi$-algebra with compatible operations and $\alg{B} \subseteq \Alg{A}$ be a non-empty subset of elements. Then
	\begin{equation}
	\pfltr{\alg{B}} = \set{\alg{a} \in \Alg{A}}{\Box\alg{b}_0 \land \dots \land \Box\alg{b}_{n-1} \leq \alg{a} \text{ for some } \alg{b}_0,\dots,\alg{b}_{n-1} \in \alg{B}}. \label{eq-pfltr}
	\end{equation}
\end{prop}
The proof is left to the reader. 

Let us note that if algebra  $\Alg{A}$ is finitely generated and $\Alg{A}/\alg{B}$ is finite, the condition \eqref{eq-pfltr} can be simplified. 

\begin{cor} \label{cor-dense-g}
	Let $\Alg{A}$ be a finitely generated $\VWSFi$-algebra with compatible operations, $\alg{B} \subseteq \Alg{A}$ be a set of elements of $\Alg{A}$ closed under $\land$ and such that $\Alg{A}/\pfltr{B}$ is finite. Then there is an element $\alg{b} \in \alg{B}$, such that
	\begin{equation}
	\pfltr{\alg{B}} = \set{\alg{a} \in \Alg{A}}{\Box\alg{b} \leq \alg{a}}. \label{eq-cor-dense}
	\end{equation}	 
\end{cor}
\begin{proof} Suppose $\Alg{A}$ ia a finitely generated $\VWSFi$-algebra with compatible operations,  $\alg{B} \subseteq \Alg{A}$ and quotient algebra $\Alg{A}/\pfltr{B}$ is finite. Recall (see e.g. \cite[Lemma 2.1]{Andreka_Jonsson_Nemeti_1991}) that finiteness of a quotient of a finitely generated algebra implies  compactness of the congruence. Thus, the congruence defined by $\pfltr{\alg{B}}$ is compact, that is filter $\pfltr{\alg{B}}$ is generated by a finite set of elements. Every filter is closed under $\land$. Hence, if a filter is generated by a finite set of elements, it is generated by a single element: the meet of generators. 
	
Suppose that $\alg{a} \in \pfltr{\alg{B}}$ generates $\pfltr{\alg{B}}$, that is 
\begin{equation}
\pfltr{\alg{B}} = \set{\alg{c} \in \Alg{A}}{\alg{a} \leq \alg{c}}. \label{eq-cor-dense-1}
\end{equation}
By Proposition \ref{pr-genfltr}, there are elements $\alg{b}_i \in \alg{B}, i < n$ such that 
\begin{equation}
\Box\alg{b}_0 \land \dots \land \Box\alg{b}_{n-1} \leq \alg{a}. \label{eq-cor-dense-2}
\end{equation}
Recall that $\Box\alg{b}_0 \land \dots \land \Box\alg{b}_{n-1} = \Box(\alg{b}_0 \land \dots \land \alg{b}_{n-1})$, and by the assumption, $\alg{B}$ is closed under $\land$. Thus, if $\alg{b} = \alg{b}_0 \land \dots \land \alg{b}_{n-1} \in \alg{B}$, and we have
	\begin{equation}
\pfltr{\alg{B}} = \set{\alg{a} \in \Alg{A}}{\Box\alg{b} \leq \alg{a}}. \label{eq-cor-dense-10}
\end{equation}	 
\end{proof}

If $\Alg{A}$ is a $\VWSFi$-algebra with compatible operations and $\alg{B} \subseteq \Alg{A}$ is a set of its dense elements, the quotient algebra $\BRed{\Alg{A}} \bydef \Alg{A}/\pfltr{\alg{B}}$ is a \Def{Boolean projection} of $\Alg{A}$. It is clear that the h-reduct of $\BRed{\Alg{A}}$ is a Boolean algebra. Moreover, if $\Alg{A}'$ is a homomorphic image of $\Alg{A}$ having Boolean h-reduct, then $\Alg{A}'$ is a homomorphic image of $\BRed{\Alg{A}}$.   

Immediately from Corollary \ref{cor-dense-g} and the fact that set of all dense elements is closed under $\land$, we have the following.

\begin{cor} \label{cor-dense}
	Let $\Alg{A}$ be a finitely generated $\VWSFi$-algebra with compatible operations, $\alg{B} \subseteq \Alg{A}$ be a set of all dense elements of $\Alg{A}$ and $\BRed{\Alg{A}}$ be finite. Then there is an element $\alg{b} \in \alg{B}$ such that
	\begin{equation}
	\pfltr{\alg{B}} = \set{\alg{a} \in \Alg{A}}{\Box\alg{b} \leq \alg{a}}. \label{eq-cor-dense0}
	\end{equation}	 
\end{cor}

The following property of finitely generated $\VWSFi$-algebras with compatible operations is crucial for our study of projective $\VWSFi$ algebras with compatible operations.

\begin{theorem} \label{th-two} Let $\Alg{A}$ be a nontrivial finitely generated $\VWSFi$-algebra with compatible operations having finite Boolean reduct. Then $\two \in \CHom\Alg{A}$ if and only if for every $\alg{a} \in \Alg{A}$
	\begin{equation}
	\Box\alg{a} \neq \Box\neg\alg{a}. \label{eq-th-two-0}
	\end{equation}
\end{theorem}
\begin{proof} Suppose $\two \in \CHom\Alg{A}$. For contradiction: assume that for some $\alg{a} \in \Alg{A}$ we have $\Box\alg{a} = \Box\neg\alg{a}$. Then, if $\varphi$ is a homomorphism of $\Alg{A}$ onto $\two$, we would have $\Box\varphi(\alg{a}) = \Box\neg\varphi(\alg{a})$ and $\varphi(\alg{a}) \in \{\zero,\one\}$, which is impossible.
	
	\subsubsection*{}
	
	Conversely, suppose that $\Alg{A}$ is a finitely generated algebra that does not contain elements satisfying \eqref{eq-th-two-0}. If $\Alg{A}$ does not contain dense elements distinct from $\one$, then h-reduct of $\Alg{A}$ is a Boolean algebra, that is, $\Alg{A} = \BRed{\Alg{A}}$ and it is finite by the assumption of the theorem. Hence, by Proposition \ref{pr-WS5}(c), $\Alg{A} \cong \Alg{A}_0 \times \dots \times \Alg{A}_{n-1}$, where $\Alg{A}_i, i <n$ are simple algebras. Thus, either one of algebras $\Alg{A}_i$ is (isomorphic to) $\two$ (and obviously $\two \in \CHom\Alg{A}$), or there is an element $\alg{a} = (\alg{a}_0,\dots,\alg{a}_{n-1}) \in \Alg{A}$ such that $\alg{a}_i \in \Alg{A}_i$ and $\zero_{\Alg{A}_i} < \alg{a}_i < \one_{\Alg{A}_i}$. Note, that because h-reducts of algebras $\Alg{A}_i$ are Boolean, we have $\zero_{\Alg{A}_i} < \neg\alg{a}_i < \one_{\Alg{A}_i}$. Therefore, by Proposition \ref{pr-WS5}(b) (and taking into account that all additional operations are compatible), $\Box\alg{a}_i = \Box\neg\alg{a}_i = \zero_{\Alg{A}_i}$, which contradicts \eqref{eq-th-two-0}. 
	
	\subsubsection*{}	
	Now, let us assume that $\Alg{A}$ has distinct from $\one$ dense elements. Let $\alg{D}$ be a set of all dense elements of $\Alg{A}$. Then, by Corollary \ref{cor-dense}, there is a dense element $\alg{d} \in \alg{D}$ such that 
	\begin{equation}
	\pfltr{\alg{D}} = \set{\alg{a} \in \Alg{A}}{\Box\alg{d} \leq \alg{a}} = \pfltr{\Box\alg{d}}. \label{eq-th-two-1}
	\end{equation}
	
	Observe that $\Box\alg{d} > 0$, for otherwise we would have $\Box\alg{d} = \zero$, and because element $\alg{d}$ is dense, $\Box\neg\alg{d} = \Box\zero = \zero$. That is, we would have  $\Box\alg{d} = \Box\neg\alg{d}$. 
	
	Recall that open elements of $\Alg{A}$ form a Boolean algebra. So, $\Box\alg{d} > 0$ yields $\neg\Box\alg{d} > \zero$ and element $\neg\Box\alg{d}$ is open, i.e. $\Box\neg\Box\alg{d} = \alg{d}$. Let us consider filter
	\begin{equation}	
	\pfltr{\neg\Box\alg{d}} = \set{\alg{a} \in \Alg{A}}{\neg\Box\alg{d} \leq \alg{a}}. \label{eq-th-two-2}
	\end{equation}
	It is clear that filters $\pfltr{\Box\alg{d}}$ and $\pfltr{\neg\Box\alg{d}}$ complement each other and
	\begin{equation}
	\Alg{A} \cong \Alg{B} \times \Alg{C} \label{eq-th-two-3}
	\end{equation}
	where 
	\begin{equation}
	\Alg{B} \bydef \Alg{A}/\pfltr{\Box\alg{d}} \text{ and } \Alg{C} \bydef \Alg{A}/\pfltr{\neg\Box\alg{d}}. \label{eq-th-two-4}
	\end{equation}
	
	Let us observe that element $\alg{c} \bydef \alg{d}/\pfltr{\neg\Box\alg{d}} \in \Alg{C}$ is a dense element, for 
	\begin{equation}
	\neg\alg{c} = \neg\alg{d}/\pfltr{\neg\Box\alg{d}} = \zero_\Alg{A}/\pfltr{\neg\Box\alg{d}} = \zero_\Alg{C} \label{eq-th-two-10}
	\end{equation} 
	and subsequently
	\begin{equation}
	\Box\neg\alg{c} = \zero_\Alg{C}. \label{eq-th-two-11}
	\end{equation}
	On the other hand,
	\begin{equation}
	\Box\alg{c} = \Box\alg{d}/\pfltr{\neg\Box\alg{d}} = \zero_\Alg{C}, \label{eq-th-two-12}
	\end{equation}
	because $\neg\Box\alg{d}/\neg\Box\alg{d} = \one_\Alg{C}$. Thus, combining \eqref{eq-th-two-11} and \eqref{eq-th-two-12}, we have
	\begin{equation}
	\Box\alg{c} = \Box\neg\alg{c} = \zero_\Alg{C}. \label{eq-th-two-13}
	\end{equation}
	
	Let us turn our attention to algebra $\Alg{B}$. Because $\pfltr{\Box\alg{d}}$ is a filter generated by all dense elements of $\Alg{A}$, algebra $\Alg{B}$ is a Boolean projection of $\Alg{A}$ and it is finite by the assumption of the theorem. Moreover, by Proposition \ref{pr-WS5}(c), $\Alg{B}$ is a direct product of simple algebras. Suppose $\Alg{B} = \Alg{B}_0 \times \dots \times \Alg{B}_{n-1}$, where $\Alg{B}, i < n$ are simple algebras. If one of the algebras $\Alg{B}_i$ is (isomorphic to) $\two$, algebra $\two$ is clearly a homomorphic image of $\Alg{A}$. All what remains to show is that, the case when every algebra $\Alg{B}_i$ is distinct from $\two$, is impossible.  
	
	Indeed, suppose that for all $i < n$, there is $\alg{b}_i \in \Alg{B}_i$ such that $\zero_{\Alg{B}_i} < \alg{b}_i < \one_{\Alg{B}_i}$. Since h-reduct of $\Alg{B}_i$ is Boolean, we have $\zero_{\Alg{B}_i} < \neg\alg{b}_i < \one_{\Alg{B}_i}$ for all $i < n$. Therefore,  by Proposition \ref{pr-WS5}(b), because simplicity of $\Alg{B}_i$,
	\begin{equation}
	\Box\alg{b}_i = \Box\neg\alg{b}_i = \zero_{\Alg{B}_i}  \label{eq-th-two-14}
	\end{equation}
	
	Let us consider element
	\begin{equation}
	\alg{a} = (\alg{b}_0,\dots,\alg{b}_{n-1},\alg{c}).  \label{eq-th-two-15}
	\end{equation}
	Immediately from \eqref{eq-th-two-13} and \eqref{eq-th-two-14} it follows that $\Box\alg{a} = \Box\neg\alg{a} = \zero$, i.e. the case when all algebras $\Alg{B}_i$ are distinct from $\two$, is impossible.
\end{proof}	

Now, let us consider some classes of algebras arising from logic that can be regarded as $\VWSFi$-algebras with compatible operations.

\subsubsection{Heyting Algebras with Regular Involution}

An algebra $\Alg{A} = (\alg{A};\land,\lor,\impl,\neg,\sim,\zero,\one)$ is a \Def{Heyting algebra with regular involution} if $\Alg{A} = (\alg{A};\land,\lor,\impl,\neg,\zero,\one)$ is a Heyting algebra and $\sim$ is an unary operation
for which the following equations are satisfied:
\begin{itemize}
	\item[(a)] $\sim (\alg{a} \lor \alg{b}) = \ \sim\alg{a} \ \land \sim \alg{b}$;
	\item[(b)] $\sim\sim \alg{a} = \alg{a}$;
	\item[(c)] $\sim\neg\alg{a} = \neg\neg\alg{a}$.
\end{itemize}

\begin{prop} \label{pr-HIfin} Let $\Alg{A}$ be a finitely generated Heyting algebra with rregular involution. If h-reduct of $\Alg{A}$ is Boolean, then $\Alg{A}$ is finite. 
\end{prop}	
\begin{proof}
	If h-reduct of the algebra is Boolean, we have $\neg\neg\alg{a} = \alg{a}$. Hence, (c) entails that $\sim\neg\alg{a} = \alg{a}$. Subsequently from (b) we get $\sim\sim\neg\alg{a} = \neg\alg{a}$ and therefore $\sim\alg{a} = \neg\alg{a}$, i.e. $\sim$ and $\neg$ coincide. Thus, $\Alg{A}$ is finite, for every finitely generated Boolean algebra is finite.
\end{proof}

Heyting algebras with regular involution form a variety denoted by $\VHRI$ (see \cite{Meskhi_Discriminator_1982}).
Because $\sim\one = \zero$, algebra $\two$ is a subalgebra of any $\VHRI$-algebra. Moreover, $\VHRI$ is a discriminator variety (see \cite[Theorem 4]{Meskhi_Discriminator_1982}).  

If $\Box$ denotes $\neg\sim$, all conditions \eqref{eq-ws5} are satisfied in every $\VHRI$-algebra. Moreover, $\VWSFi$-congruences coincide with $\VHRI$-congruences. Thus, $\VHRI$-algebras can be viewed as $\VWSFi$-algebras with compatible operation $\sim$. 

\subsubsection{Heyting Algebras with Dual Pseudocomplement} An algebra $\Alg{A} = (\alg{A};\land,\lor,\impl,\zero,\one,\dneg)$ is said to be a \Def{Heyting algebra with Dual Pseudocomplement}, if $\Alg{A} = (\alg{A};\land,\lor,\impl,\zero,\one)$ is a Heyting algebra and $\dneg$ is a dual pseudocomplement, that is for each $\alg{a},\alg{b} \in \alg{A}$ we have 
\begin{equation}
\alg{a} \lor \alg{b} = \one \text{ if and only if } \alg{b} \ge \dneg\alg{a}. \label{eq-dualneg}
\end{equation}

\begin{prop} \label{pr-HDPfin} Let $\Alg{A}$ be a finitely generated Heyting algebra with pseudocomplement. If h-reduct of $\Alg{A}$ is Boolean, then $\Alg{A}$ is finite. 
\end{prop}	
\begin{proof}
Recall from \cite[1.3(26)]{Rauszer_Semi-Boolean_1973} that $\neg\alg{a} \leq \dneg\alg{a}$. If h-reduct of the algebra is Boolean, we have $\alg{a} \lor \neg\alg{a} = \one$. Hence, \eqref{eq-dualneg} entails that $\neg\alg{a} \ge \dneg\alg{a}$ and we have $\neg\alg{a} = \dneg\alg{a}$, i.e. $\sim$ and $\neg$ coincide. Therefore $\Alg{A}$ is finite, for every finitely generated Boolean algebra is finite.
\end{proof}

A set of all Heyting algebras with dual pseudocomplement forms a variety denoted by $\HDP$ (see \cite{Sanka_Heyting_1985}).

Clearly, $\two$ is a subalgebra of any nontrivial $\HDP$-algebra.

Let $\boxdot \alg{a} \bydef \neg\dneg \alg{a}$, and we define $\boxdot^0 \alg{a} \bydef \alg{a};$ and $\boxdot^{n+1} \alg{a} \bydef \boxdot\boxdot^n x$. Consider varietyes $\HDP_n, n \ge 0$ defined by identities  (see e.g. \cite{Sanka_Heyting_1985}) 

\[
\boxdot^{n+1} x \approx \boxdot^n x.
\]
We can define a new operator
\[
\Box \alg{a} \bydef \bigwedge_{i \leq n} \boxdot^i \alg{a}.
\]
Then in $\HDP_n$ all equalities \eqref{eq-ws5} hold (see \cite{Sanka_Heyting_1985}), thus we can view $\HDP_n$- algebras as $\VWSFi$-algebras with $\sim$ being an additional compatible operation. 

Recall also from \cite{Sanka_Heyting_1985} that for every $n > 0$, variety $\HDP_n$ is discriminator.
\subsubsection{Double Heyting Algebras}
\Def{Double Heyting algebra} is an algebra $\Alg{A} = (\alg{A};\land,\lor,\impl,\dimpl,\zero,\one)$, where $\Alg{A} = (\alg{A};\land,\lor,\impl,\zero,\one)$ is a Heyting algebra and $\dimpl$ is a dual relative pseudocomplement, that is for each $\alg{a},\alg{b},\alg{c} \in \alg{A}$ we have 
\[
\alg{a} \lor \alg{b} = \alg{c} \text{ if and only if } \alg{b} \ge \alg{c} \dimpl \alg{a}.
\] 

A set of all double Heyting algebras forms a variety denoted by $\DHt$ (see e.g. \cite{Taylor_C_Discriminator_2016}. It is not hard to see that $\dneg \alg{a} = \one \dimpl \alg{a}$. 

\begin{prop} \label{pr-DHfin} Let $\Alg{A}$ be a finitely generated double Heyting algebra. If h-reduct of $\Alg{A}$ is Boolean, then $\Alg{A}$ is finite. 
\end{prop}	
\begin{proof} Recall from Proposition \ref{pr-HDPfin} that if h-reduct of $\Alg{A}$ is Boolean, operations $\neg$ and $\dneg$ coincide. Hence, the following holds 
	\[
	\begin{array}{lll}
	1 &\alg{a} \dimpl \alg{b} \leq \dneg(\alg{a} \impl \alg{b}) & \text{\cite[1.3(29)]{Rauszer_Semi-Boolean_1973}}\\
	2 & \alg{a} \dimpl \alg{b} \leq \neg(\alg{a} \impl \alg{b}) & \text{from 1, because }\Alg{A} \text{ is Boolean and } \dneg = \neg \\ 
	3 & \neg(\alg{a} \dimpl \alg{b}) \leq \alg{a} \impl \alg{b} & \text{\cite[1.3(30)]{Rauszer_Semi-Boolean_1973}}\\
	4 & \neg(\alg{a} \impl \alg{b}) \leq \neg\neg(\alg{a} \dimpl \alg{b}) & \text{from 3}\\
	5 & \neg(\alg{a} \impl \alg{b}) \leq (\alg{a} \dimpl \alg{b}) & \text{from 4, because h-reduct of }  \Alg{A} \text{ is Boolean}\\
	6 & \alg{a} \dimpl \alg{b} = \neg(\alg{a} \impl \alg{b})  & \text{from 3 and 5} 
	\end{array}
		\]
Thus, in $\Alg{A}$ both additional operations are expressed via Boolean operations. Therefore $\Alg{A}$ is a finitely generated as Boolean algebra and it is finite.
\end{proof}

Clearly, $\two$ is a subalgebra of any nontrivial $\DHt$-algebra. Let us observe that for the same reason as for $\HDP$-algebras, $\DHt_n$-algebra can be viewed as a $\WSFi$-algebra with compatible operations  (see e.g. \cite{Sanka_Heyting_1985}). 

Recall also (see \cite{Sanka_Heyting_1985}) that for every $n > 0$, variety $\DHt_n$ is discriminator.

\section{Projective Algebras} \label{sec-proj}

In this section we study projective algebras in the varieties in which compact congruences are factor congruences.

\subsection{Retracts}
	
\begin{definition}
	An algebra $\Alg{B}$ is a \Def{retract} of an algebra $\Alg{A}$ if there are homomorphism $\varphi:\Alg{A} \mapping \Alg{B}$ and $\psi:\Alg{B} \mapping \Alg{A}$ such that $\varphi\circ\psi: \Alg{B} \mapping \Alg{B}$ is the identity. Homomorphism $\psi$ is called \Def{injection}, and homomorphism $\varphi$ is called \Def{retraction}.
\end{definition}

\begin{theorem} \label{th-prod-retact}
	Let $\Alg{A} = \Alg{B} \times \Alg{C}$. Then $\Alg{B}$ is a retract of $\Alg{A}$ if and only if there is a homomorphism $\chi:\Alg{B} \mapping \Alg{C}$. 
\end{theorem}
\begin{proof} Because $\Alg{A} = \Alg{B} \times \Alg{C}$, we can view elements of $\Alg{A}$ as pairs $(\alg{b},\alg{c})$, where $\alg{b} \in \Alg{B}$ and $\alg{c} \in \Alg{C}$. Let
	\begin{equation}
\varphi: (\alg{b},\alg{c}) \mapping \alg{b}, \label{eq-th-prod-retact-1}
\end{equation} 
that is, $\varphi$ is a projection, and therefore, $\varphi$ is a homomorphism of $\Alg{A}$ onto $\Alg{B}$. Our goal is to construct a homomorphism $\psi:\Alg{B} \mapping \Alg{A}$ in such a way that $\varphi$ is a retraction and $\psi$ is an injection.	

First, note that if $\Alg{C}$ is a trivial algebra, then $\Alg{B} \cong \Alg{A}$ and $\Alg{B}$ trivially is a retract of $\Alg{A}$.	
	
Suppose that $\Alg{C}$ is not trivial and $\chi:\Alg{B} \mapping \Alg{C}$ is a homomorphism. Let
	\begin{equation}
	\psi: \alg{b} \mapping (\alg{b},\chi(\alg{b})) \label{eq-th-prod-retact-2}.
	\end{equation} 
	It is clear that $\varphi \circ\psi$ is the identity. Let us verify that $\psi$ is a homomorphism.
	
	Suppose $f$ is an $n$-ary fundamental operation and $\alg{b}_0,\dots,\alg{b}_{n-1} \in \Alg{B}$. Then, taking into account that $\chi$ is a homomorphism, we have
	\begin{equation}
	\begin{split}
	\psi(f(\alg{b}_0,\dots,\alg{b}_{n-1})) & = (f(\alg{b}_0,\dots,\alg{b}_{n-1}),\ \chi(f(\alg{b}_0,\dots,\alg{b}_{n-1})))\\
	& = (f(\alg{b}_0,\dots,\alg{b}_{n-1}),\ f(\chi(\alg{b}_0),\dots,\chi(\alg{b}_{n-1}))))\\
	& = f((\alg{b}_0,\ \chi(\alg{b}_0)),\dots, (\alg{b}_{n-1},\ \chi(\alg{b}_{n-1}))) \\
	& = f(\psi(\alg{b}_0),\dots,\psi(\alg{b}_{n-1})).
	\end{split}
	\end{equation}
	Thus, $\psi$ is a homomorphism, and $\Alg{B}$ is a retract of $\Alg{A}$.
	
	Conversely, suppose that $\Alg{B}$ is a retract of $\Alg{A}$ and $\varphi,\psi$ are respectively retraction and injection. Thus, $\varphi\circ\psi$  is identity. Let us consider a map
	\begin{equation}
	\chi: \alg{b} \mapping \pi_2(\psi(\alg{b})), \label{eq-th-pair-retract-10}
	\end{equation}
	where $\pi_2$ is a projection on the second component. In other words, we let
	\begin{equation}
	\chi: \alg{b} \overset{\psi}{\mapping} \pair{\alg{b}}{\alg{c}}  \overset{\pi_2}{\mapping} \alg{c}. \label{eq-th-pair-retract-11}
	\end{equation}
	Because operations on elements of a direct product being performed independently in each component, it is not hard to see that $\chi$ is indeed a homomorphism. 	
\end{proof}	

Let us recall the definition of projective algebra.

\begin{definition} An algebra $\Alg{A}$ is \Def{projective} in a class of algebras $\Var{K}$  if for every pair of algebras $\Alg{B},\Alg{C} \in \Var{K}$, every homomorphism $\varphi:\Alg{A} \mapping \Alg{C}$, and every homomorphism $\psi: \Alg{B} \mapping \Alg{C}$ of $\Alg{B}$ onto $\Alg{C}$, there exists a homomorphism $\chi: \Alg{A} \mapping \Alg{B}$ such	that $\psi\circ\chi = \varphi$. An algebra $\Alg{A}$ is projective in a variety $\var$ if and only if $\Alg{A}$ is a retract of a free algebra.
\end{definition}

\subsection{Minimal Algebras: Necessary Conditions of Projectivity}

In this section we give a simple necessary condition of projectivity in the varieties that contain minimal algebras. In the following section we study when this necessary condition becomes sufficient.

\subsubsection{Minimal Algebras}

A subalgebra $\Alg{A}'$ of an algebra $\Alg{A}$ is said to be \Def{proper}, if $\Alg{A} \setminus \Alg{A}' \neq \emptyset$.
A non-degenerate algebra $\Alg{A}$ is \Def{minimal}, if $\Alg{A}$ does not contain proper subalgebras (see \cite{Quackenbush_Demi-semi-primal_1971}). Clearly, a non-degenerate algebra is minimal if and only if it is generated by any of its elements. Hence, minimal algebras are countable. If $\var$ is a variety, by $\mvar$ we denote the set of all minimal algebras from $\var$.	   

\begin{example}
	In a variety of monadic Heyting algebras a non-degenerate algebra is minimal if and only if it is generated by the unit element. In particular, there is the only minimal monadic Heyting algebra, namely $\two$. 
\end{example}

\begin{example} $n$-valued Post algebras are minimal for each $n >1$ (see \cite{Burris_Discriminator_1992}).  
\end{example}

Let us consider some properties of minimal algebras.

\begin{prop}
	Let $\var$ be a variety and $\MA$ be a minimal subalgebra of $\Falgv{\omega}$. Then every minimal algebra from $\var$ is a homomorphic image of $\MA$. 
\end{prop}
\begin{proof} Suppose $\Alg{A} \in \var$ is at most countable minimal algebra. Then there is a homomorphism $\varphi:\Falgv{\omega} \mapping \Alg{A}$ that maps $\Falgv{\omega}$ onto $\Alg{A}$. Because $\MA$ is a subalgebra of $\Falgv{\omega}$, the restriction $\varphi'$ of $\varphi$ to $\MA$ is a homomorphism of $\MA$ to $\Alg{A}$. Since $\Alg{A}$ is minimal, $\varphi'$ maps $\MA$ onto $\Alg{A}$.
\end{proof}

\begin{cor} Let $\var$ be a variety and $\MA$ be a finite minimal subalgebra of $\Falgv{\omega}$. Then all minimal algebras from $\var$ are finite and $\Falgv{\omega}$ contains unique up to isomorphism minimal subalgebra.
\end{cor}	

\subsubsection{ms-Full Varieties}

Often, especially when a variety has the only minimal algebra, every nontrivial algebra $\Alg{A}$ of a variety $\var$ contains at least one minimal subalgebra:
\begin{equation}
\CSub\Alg{A} \cap \mvar \neq \emptyset \label{SMin}. \tag{MS}
\end{equation}

\begin{definition} We say that a variety $\var$ is \Def{ms-full} if every nontrivial algebra from $\var$ has at least one minimal subalgebra. 
\end{definition}

It is not hard to see that if no nontrivial algebra from a variety $\var$ contains a trivial subalgebra and $\Falgv{\omega}$ contains a finite subalgebra, then $\var$ is ms-full.

Recall (see \cite{Blk_Pgz_3}) that a \Def{constant term} of $\Alg{A}$ or a variety $\var$ is any nullary or constant unary term function, or less precisely, the element of $\Alg{A}$ (or of each member of $\var$) that constitutes the range of such a function. A variety $\var$ of algebras is \Def{pointed} if it has at least one constant term, \Def{double-pointed} if it has at least two constant terms that are distinct in any nontrivial member of $\var$.

\begin{example}
	If no nontrivial algebra from a pointed variety $\var$ has a trivial subalgebra, $\var$ is ms-full: any nontrivial algebra from $\var$ contains a minimal subalgebra generated by its constant. Every double pointed variety $\var$ is ms-full, for every nontrivial algebra from $\var$ contains a minimal subalgebra generated by the values of constant terms. For instance, varieties of Heyting or monadic Heyting algebras are double pointed, and hence, they are ms-full.
\end{example}

\subsubsection{Necessary condition of Projectivity} The following simple proposition plays an important role in what follows.

\begin{prop} \label{pr-sub-nontr}
	If a minimal algebra $\MA$ is a homomorphic image of algebra $\Alg{A}$, then  $\MA$ is a homomorphic image of any nontrivial subalgebra of $\Alg{A}$. 
\end{prop}
\begin{proof} Let $\varphi:\Alg{A} \mapping \MA$ be a homomorphism of $\Alg{A}$ onto $\MA$ and $\Alg{B} \leq \Alg{A}$ be a subalgebra of $\Alg{A}$. If $\varphi'$ is a restriction of $\varphi$ to $\Alg{B}$, we have 
	\[
	\varphi'(\Alg{B}) \leq \varphi(\Alg{A}) = \MA,
	\] 
	and	$\varphi'(\Alg{B}) = \MA$ because $\MA$ has no proper subalgebras.
\end{proof}	

\begin{cor} 
	Let $\var$ be a variety and $\MA \in \var$ be a minimal algebra. Then $\MA$ is a homomorphic image of every projective in $\var$ algebra.
\end{cor}
\begin{proof} If an algebra $\Alg{A}$ is projective, $\Alg{A}$ is a subalgebra of a free algebra. Without loosing generality we can assume that $\Alg{A}$ is a subalgebra of a free algebra of some infinite rank. Because minimal algebras are countable, every minimal algebra is a homomorphic image of any free algebra of infinite rank. 
\end{proof}

\begin{definition}
	Let $\var$ be a variety containing minimal algebras. An algebra $\Alg{A} \in \var$ is \Def{mh-full} if $\Alg{A}$ has every algebra from $\mvar$ as a homomorphic image, that is if 
	\begin{equation}
	\mvar \subseteq \CHom\Alg{A}. \label{HMin} \tag{HF}
	\end{equation} 
\end{definition}

Let us observe that, because each minimal algebra being countable, it is a homomorphic image of $\Falgv{\omega}$ and consequently $\Falgv{\omega}$ is always mh-full, provided $\var$ contains minimal algebras. Also,  Proposition \ref{pr-sub-nontr} entails the following.

\begin{prop} \label{pr-mhfull}
	Every nontrivial subalgebra of mh-full algebra is mh-full.
\end{prop}

Now, we can give the following necessary condition of projectivity.

\begin{cor} \label{cor-neces} (comp \cite[Theorem 5.1]{Quackenbush_Demi-semi-primal_1971}).
	Let $\var$ be a variety. Then every nontrivial projective algebra $\Alg{A} \in \var$ is mh-full.
\end{cor}
The proof follows immediately from Proposition \ref{pr-sub-nontr}.

\begin{example} \label{ex-sinfglema} Variety $\VMHA$ of all monadic Heyting algebras contains the only minimal algebra: $\two$. Hence, every projective $\VMHA$-algebra must have $\two$ as its homomorphic image. Thus, every finite projective $\VMHA$-algebra must have at least one open atom.
\end{example}	

Theorem \ref{th-prod-retact} has another useful corollary.

\begin{cor} \label{cor-retract}
	Let $\var$ be an ms-full variety and let $\Alg{A} = \Alg{B} \times \Alg{C}$, where $\Alg{B},\Alg{C} \in \var$. If $\Alg{B}$ is mh-full, then $\Alg{B}$ is a retract of $\Alg{A}$. 
\end{cor}
\begin{proof}  Indeed, if $\Alg{C}$ is trivial, $\Alg{B}$ is a retract of $\Alg{A}$. If $\Alg{C}$ is nontrivial, then, by the assumption, $\Alg{C}$ contains a minimal subalgebra $\MA$. Also by the assumption, $\MA$ is a homomorphic image of $\Alg{B}$. Application of Theorem \ref{th-prod-retact} completes the proof.
\end{proof}

\subsection{Projectivity}

If $\theta$ is a factor congruence on an algebra $\Alg{A}$, we can re-phrase Theorem \ref{th-prod-retact} for factor congruences.

\begin{theorem} \label{th-pair-retract}
	Let $\Alg{A}$ be an algebra and $\theta,\theta'$ be a pair of factor congruences on $\Alg{A}$. Then $\Alg{A}/\theta$ is a retract of $\Alg{A}$ if and only if there is a homomorphism $\chi:\Alg{A}/\theta \mapping \Alg{A}/\theta'$. 
\end{theorem}

\begin{cor} \label{cor-proj}
	Let $\var$ be an ms-full variety and let $\theta$ be a factor congruence on $\Falgv{n}$ such that $\Falgv{n}/\theta$ is nontrivial. Then $\Falgv{n}/\theta$ is projective if and only $\Falgv{n}/\theta$ is mh-full. 
\end{cor}
\begin{proof} Suppose $\theta$ is a factor congruence on $\Falgv{n}$ and $\Falgv{n}/\theta$ is not trivial. 
	
If  $\Falgv{n}/\theta$ is projective, then $\Falgv{n}/\theta$ is a nontrivial subalgebra of $\Falgv{\omega}$ which is mh-full. Hence, by Proposition \ref{pr-mhfull},  $\Falgv{n}/\theta$ is mh-full too.
	
Conversely, suppose $\Falgv{n}/\theta$ is mh-full. By the assumption, $\theta$ is a factor congruence, hence there is a complement congruence $\theta'$ and $\Falgv{n} = \Falgv{n}/\theta \times \Falgv{n}/\theta'$. 

If $\Falgv{n}/\theta'$ is trivial, then $\Falgv{n} = \Falgv{n}/\theta$, and therefore, $\Falgv{n}/\theta$ is projective. 

If $\Falgv{n}/\theta'$ is not trivial, then, because $\var$ is ms-full, $\Falgv{n}/\theta'$ has a minimal subalgebra $\MA$. Since $\Falgv{n}/\theta$ is mh-full, $\MA$ is a homomorphic image of $\Falgv{n}/\theta$. That is, there is a homomorphism  $\chi: \Falgv{n}/\theta \mapping \Falgv{n}/\theta'$, so we can apply Theorem \ref{th-pair-retract} and conclude that $\Falgv{n}/\theta$ is a retract of $\Falgv{n}$, that is, $\Falgv{n}/\theta$ is projective.
\end{proof}

Let us recall that finitely presented algebras are exactly factors of free algebras of finite rank by compact congruences. Hence, from the above Corollary we get:

\begin{cor} \label{cor-fp-proj}
	Let $\var$ be an ms-full variety in which each compact congruence is a factor congruence. Then a nontrivial finitely presented algebra $\Alg{A} \in \var$ is projective if and only it is mh-full. 
\end{cor}

Recall also that in discriminator varieties all compact congruences are principal and all principal congruences are factor congruences (see Proposition  \ref{pr-discr}). 
\begin{cor} \label{cor-disc-proj}
	Let $\var$ be an ms-full discriminator variety. Then a finitely presented algebra $\Alg{A} \in \var$ is projective if and only if every minimal algebra is a homomorphic image of $\Alg{A}$. 
\end{cor}

Often a variety contains just one minimal algebra: for instance, varieties of Heyting algebras or monadic Heyting algebras. In this case the above corollary can be stated as follows.

\begin{cor} \label{cor-disc-unique-single-proj}
	Let $\var$ be a discriminator variety containing a single minimal algebra $\MA$ which embeds in every nontrivial algebra. Then, a finitely presented algebra $\Alg{A}$ is projective if and only if $\MA$ is a homomorphic image of $\Alg{A}$. 
\end{cor}

Combining Proposition \ref{pr-sub-nontr} and Corollary \ref{cor-disc-unique-single-proj} we obtain the following.

\begin{cor} \label{cor-disc-unique-proj}
	Let $\var$ be a discriminator variety containing a single minimal algebra $\MA$ which embeds in every nontrivial algebra. Then any finitely presented subalgebra of $\Falgv{\omega}$ is projective. 
\end{cor}
\begin{proof} Suppose a finitely presented algebra $\Alg{A}$ is a subalgebra of $\Falgv{\omega}$. Then, because $\Alg{A}$ is finitely generated, $\Alg{A}$ is a subalgebra of $\Falgv{n}$ for some $n \in \omega$. By Proposition \ref{pr-sub-nontr}, $\MA \in \CHom\Alg{A}$. By Corollary \ref{cor-disc-unique-single-proj}, $\Alg{A}$ is projective. 
\end{proof}

Let us apply the obtain criterion of projectivity. For this, recall that varieties $\VWSFi$, $\VHRI$ and $\HDP_n$ are discriminator varieties and have a minimal $\two$ that embeds in every nontrivial algebra. 

\begin{theorem} \label{th-crit-appl}
	Let variety $\var$ be a subvariety of $\VWSFi$ (or $\VHRI$, or $\HDP_n$, or $\DHt_n$). Then a finitely presented in $\var$ algebra is projective if and only if it has two-element algebra as a homomorphic image. 
\end{theorem}

\section{Primitive Quasivarieties} \label{sec-prim}

A quasivariety $\qvar$ is said to be \Def{primitive} (see e.g.  \cite{GorbunovBookE}) if every subquasivariety $\qvar_{\var_0} \subseteq \qvar$ is a relative variety, that is, $\qvar_0 = \qvar \cap \var$ for some variety $\var$.

Given a variety $\var$, by $\varq$ we denote the quasivariety generated by $\Falgv{\omega}$, thus, $\varq$ is the smallest quasivariety equational closure of which is $\var$. It is clear that for any primitive quasivariety $\qvar$, we have $\qvar = \varq$, where $\var$ is the equational closure of $\qvar$.

\begin{prop} \label{pr-prim}
	Let $\var$ be a variety. Then $\varq$ is primitive if and only if for every subvariety $\var_0 \subseteq \var$ such that $\falg{\omega}{\var_0} \in \varq$
	\begin{equation}
	\var_0 \cap \varq = \qvar_{\var_0}. \label{eq-pr-prim-0}
	\end{equation}
\end{prop}
\begin{proof}
If $\falg{\omega}{\var_0} \in \varq$, we have $\falg{\omega}{\var_0} \in \var_0 \cap \varq$ and consequently $\qvar_{\var_0} \subseteq \var_0 \cap \varq \subseteq \var_0$. Any subquasivariety of a primitive quasivariety is primitive. Hence, quasivariety $\var_0 \cap \varq$ is primitive and $\qvar_{\var_0} = \var_0 \cap \varq$.

Conversely, suppose that \eqref{eq-pr-prim-0} holds for every subvariety $\var_0 \subseteq \var$ such that $\falg{\omega}{\var_0} \in \varq$. Let $\qvar \subseteq \varq$ be a subquasivariety and $\var_0$ be a variety generated by $\qvar$. Clearly, $\falg{\omega}{\var_0} \in \qvar \subseteq \varq$. Hence, \eqref{eq-pr-prim-0} holds and $\qvar$ is a relative variety. Thus, quasivariety $\varq$ is primitive.
\end{proof}
	
For the duration of this section we assume that $\var$ is a variety such that 
\begin{equation} \label{eq-cond} \tag{FC}
\begin{array}{ll} 
\text{(a)} & \var \text{ is ms-full}; \\
\text{(b)} & \text{Every compact congruence of any } \var\text{-algebra is principal};\\
\text{(c)} & \text{Every principal }\var\text{-congruence  is a factor congruence}.
\end{array} 
\end{equation}

\begin{example} Any double pointed discriminator variety satisfies the above conditions. For instance,
	$\VWSFi$ contains minimal algebra $\two$ and every $\VWSFi$-algebra has $\two$ as a subalgebra.
\end{example}
	
Let $\varmh \bydef \set{\Alg{A} \in \var}{ \Alg{A} \text{ is mh-full}}$. 

\begin{prop} $\varmh$  is a prevariety. \label{pr-prevar}
\end{prop}
\begin{proof} It is clear that $\CProd\varmh \subseteq \varmh$, for if $\Alg{B}$ is a homomorphic image of $\Alg{A}$, then $\Alg{B}$ is a homomorphic image of $\Alg{A} \times \Alg{C}$ for any algebra $\Alg{C}$.
	
	Because all algebras from $\mvar$ are minimal, immediately from Proposition \ref{pr-sub-nontr} we conclude that $\CSub\varmh \subseteq \varmh$.
\end{proof}	 

Note that $\varmh$ is a quasivariety if and only if $\varmh$ can be axiomatized by a set of first-order sentences (see Proposition \ref{pr-prev-quas}). 

\begin{theorem}
If $\varmh$ is a quasivariety, then $\varmh = \varq$. Moreover, if $\mvar$ contains just one algebra, then $\varq$ is primitive.
\end{theorem} \label{th-prim}
\begin{proof} First, let us observe that, since every minimal algebra is countable, $\Falgv{\omega} \in \varmh$. Hence
	\begin{equation}
	\varq \subseteq \varmh. \label{eq-th-prim-0}
	\end{equation}
	Let us prove that $\varmh \subseteq \varq$. By Proposition \ref{pr-qfpgen}, quasivariety $\varmh$ is generated by all finitely $\varmh$-presented algebras. Thus, to prove the claim, it is sufficient to show that every finitely $\varmh$-presented algebra $\Alg{A}$ is in $\varq$. 
	
	Indeed, assume that algebra $\Alg{A}$ is a nontrivial finitely $\varmh$-presented and $\pair{X}{\Delta}$ is a defining pair, that is, $\Alg{A} \cong \falg{X,\Delta}{\varmh}$ and let us consider $\falg{X,\Delta}{\var}$. Then, taking into accounts that $\varmh \subseteq \var$, by Proposition \ref{pr_fp}(a), we can conclude that $\falg{X,\Delta}{\varmh}$ is a homomorphic image of $\falg{X,\Delta}{\var}$. Hence, every homomorphic image of $\falg{X,\Delta}{\varmh}$ is a homomorphic image of $\falg{X,\Delta}{\var}$. Recall that $\falg{X,\Delta}{\varmh} \in \varmh$, thus, $\falg{X,\Delta}{\var} \in \varmh$.
	
	Since $\falg{X,\Delta}{\var}$ has nontrivial homomorphic images, it is nontrivial itself.  Moreover, $\falg{X,\Delta}{\var}$ is a quotient algebra of $\falg{n}{\var}$ (where $n$ is cardinality of $X$) by a compact congruence $\theta$ (generated by pairs of elements corresponding to relations from $\Delta$). By \eqref{eq-cond}(b), congruence $\theta$ is principal, and by \eqref{eq-cond}(c), congruence $\theta$ is a factor congruence. Therefore, by Corollary \ref{cor-proj}, algebra $\falg{n}{\var}/\theta$, and hence, algebra   $\falg{X,\Delta}{\var}$, are projective in $\var$. Every projective in $\var$ countable algebra is a subalgebra of $\Falgv{\omega}$, thus, $\falg{X,\Delta}{\var}$ is a subalgebra of $\Falgv{\omega}$ and we know that $\Falgv{\omega} \in \varq$. Therefore $\falg{X,\Delta}{\var} \in \varq$.  \\
	
	Now, suppose $\var$ contains a single minimal algebra $\MA$. Then, by assumption of ms-fullness, $\MA$ is a subalgebra of every notrivial algebra from $\var$, hence, $\MA$ belongs to every non-degenerate subvariety of $\var$. 
	
	Let $\var'$ be a nontrivial subvariety of $\var$ and let us consider quasivariety $\qvar' \bydef \var' \cap \varq$. It is clear that $\qvar'$ consists of all algebras $\Alg{A}$ from $\var'$ that have $\MA$ as a homomorphic image, i.e. $\qvar' = \varmh'$. Thus, we can apply to $\var'$ the same argument that we used for $\var$ and conclude that $\qvar'= \varmh' = \qvar_{\var'}$, that is, we have 
	\[
	\qvar_{\var'} = \var' \cap \varq
	\]
	Let us note that, because $\var$ being nontrivial, $\MA \in \var'$ and subsequently $\falg{\omega}{\var'} \in \varmh' = \qvar' \subseteq \qvar$. So, by Proposition \ref{pr-prim}, we have that quasivariety $\varq$ is primitive.
\end{proof}

In case when $\var$ is a discriminator variety containing only finite minimal algebras, we can say even more. 

\begin{theorem} \label{th-axiomat}
Let $\var$ be a discriminator variety and $\MA \in \var$ be a finite minimal algebra. Then there is a first order formula $\alpha$ such that
\begin{equation}
\Alg{A} \models \alpha \text{ if and only if } \MA \in \CHom\Alg{A}. \label{eq-th-axiomat-0}
\end{equation}
\end{theorem} 
\begin{proof} Let us recall \cite[\S 10 Corollary 3]{MaltsevBook} that for each finite algebra $\Alg{B}$, if $\{\alg{b}_0,\dots,\alg{b}_{n-1}\}$ is a universe of $\Alg{B}$, there is a formula $\delta(z_{\alg{b}_0},\dots,z_{\alg{b}_{n-1}},z)$ - a diagram formula\footnote{See also a notion of open diagram in \cite{Burris_Sanka}.} of $\Alg{B}$ - such that for any algebra $\Alg{C}$,
\begin{equation}
\Alg{C} \models \exists(z_{\alg{b}_0},\dots,z_{\alg{b}_{n-1}})\forall(z)\delta(z_{\alg{b}_0},\dots,z_{\alg{b}_{n-1}},z) \text{ if and only if } \Alg{B} \cong \Alg{C}.  \label{eq-th-axiomat-1}
\end{equation}	 
Let $\Alg{B} = \MA$ and $\beta \bydef \exists(z_{\alg{b}_0},\dots,z_{\alg{b}_{n-1}})\forall(z)\delta(z_{\alg{b}_0},\dots,z_{\alg{b}_{n-1}},z)$.

Recall also from \cite[Theorem 2.8]{Blk_Pgz_3} that the discriminator term $t$ satisfies the following condition: for every $\Alg{A} \in \var$ and every $\alg{a},\alg{b},\alg{c},\alg{d} \in \Alg{A}$,
\begin{equation}
t(\alg{a},\alg{b},\alg{c}) = t(\alg{a},\alg{b},\alg{d}) \text{ if and only if } \pclass{\alg{c}}{\theta(\alg{a},\alg{b})} = \pclass{\alg{d}}{\theta(\alg{a},\alg{b})}. \label{eq-th-axiomat-2}
\end{equation}  	 

Let $\beta^t(x,y)$ be a formula obtained from formula $\beta$ by replacing every equality $r = s$ occurring in $\beta$ with
\begin{equation}
 t(x,y,r) = t(x,y,s).
\end{equation}
From \eqref{eq-th-axiomat-1} and \eqref{eq-th-axiomat-2} it follows that for any given elements $\alg{a},\alg{b} \in \Alg{A}$,  
\begin{equation}
\Alg{A} \models \beta^t(\alg{a},\alg{b}) \text{ if and only if } \Alg{A}/\theta(\alg{a},\alg{b}) \cong \MA. \label{eq-th-axiomat-3}
\end{equation}
Hence, if $\alpha \bydef \exists(x,y)\beta^t(x,y)$, then
\begin{equation}
\Alg{A} \models \alpha \text{ if and only if there are } \alg{a},\alg{b} \in \Alg{A} \text{ such that }  \Alg{A}/\theta(\alg{a},\alg{b}) \cong \MA. \label{eq-th-axiomat-4}
\end{equation}
Thus, if $\Alg{A} \models \alpha$, we have $\MA \in \CHom\Alg{A}$, and we need to prove only that if $\MA \in \CHom\Alg{A}$, then $\Alg{A} \models \alpha$. Let us observe that, because $\alpha$ contains just a finite number of variables, $\Alg{A} \models \alpha$ if and only if for every finitely generated subalgebra $\Alg{B} \leq \Alg{A}$, we have $\Alg{B} \models \alpha$. On the other hand, from Proposition \ref{pr-sub-nontr} we know that $\MA\in \CHom\Alg{A}$ yields $\MA \in \CHom\Alg{B}$ for every nontrivial subalgebra of $\Alg{A}$. Hence, if we prove that for every finitely generated algebra $\Alg{B}$, if $\MA \in \CHom\Alg{B}$, then $\Alg{B} \models \alpha$, we will complete the proof.

Let $\Alg{B}$ be a finitely generated algebra and $\MA \in \CHom\Alg{B}$. Then, because $\MA$ is finite, there is a compact congruence $\theta$ on $\Alg{B}$ such that $\MA \cong \Alg{B}/\theta$ (see e.g. \cite[Lemma 2.1]{Andreka_Jonsson_Nemeti_1991}). Because variety is discriminator, by Proposition \ref{pr-discr}(c), every compact congruence on $\Alg{B}$ is principal. Therefore, there are elements $\alg{a},\alg{b} \in \Alg{B}$ such that $\Alg{B}/\theta(\alg{a},\alg{b}) \cong \MA$. Hence, by \eqref{eq-th-axiomat-4}, $\Alg{B} \models \alpha$.
\end{proof}

\begin{cor}
	Let $\var$ be a discriminator variety such that $\var_{min}$ contains only finite algebras. Then $\var_{mh}$  is a quasivariety.
\end{cor}
\begin{proof} We already know (see Proposition \ref{pr-prevar}) that $\var_{mh}$ is a prevariety. Hence, by Proposition \ref{pr-prev-quas}, it is enough to establish that $\var_{mh}$ is axiomatizable, which immediately follows from the above theorem. 
\end{proof}

The following corollary gives a description of all primitive subquasivarieties of discriminator variety $\var$ provided $\var$ has a finite minimal algebra that embeds in any nontrivial algebra from $\var$. 

\begin{cor} Let $\var$ be a discriminator variety that has a finite minimal algebra which embeds in any nontrivial algebra from $\var$. Then a subquasivariety $\qvar \subseteq \var$ is primitive if and only if $\qvar = \qvar_{\var'}$ for some subvariety $\var' \subseteq \var$.
\end{cor}

\section{Applications} \label{sec-appl}

The goal of this section to prove the following theorem that describes projective finitely-presented algebras an some discriminator varieties that arise from logic. \\

\textbf{Till the end of this section $\var$ is a discriminator variety of $\VWSFi$-algebras with compatible operations such that unit and bottom elements of $\Falgv{\omega}$ form a subalgebra denoted by
$\two$.} 

Note that $\var$ is ms-full, for every nontrivial algebra $\Alg{A} \in \var$ there is a homomorphism from $\Falgv{\omega}$ to $\Alg{A}$ that sends top and bottom elements of $\Falgv{\omega}$ to top and bottom elements of $\Alg{A}$. It is also clear that $\two$ is the only minimal algebra of $\var$. 

\begin{theorem} \label{th-FP-Proj} Let $\Alg{A}$ be a nontrivial finitely presented in $\var$ algebra. Then the following is equivalent
	\begin{itemize}
		\item[(1)] $\Alg{A}$ is projective in $\var$;
		\item[(2)] $\two$ is a homomorphic image of $\Alg{A}$; 
		\item[(3)] $\Alg{A}$ does not contain an element $\alg{a}$ such that 
		\begin{equation}	
		\Box\alg{a} = \Box\neg\alg{a}; \label{eq-th-FP-Proj-c}
		\end{equation}	
		\item[(4)] The quasiidentity 
		\begin{equation}
		\rho \bydef \neg\Box x \land \neg\Box\neg x \Rightarrow \zero \label{eq-th-FP-Proj-q} 	
		\end{equation}	holds on $\Alg{A}$.
	\end{itemize}	
\end{theorem}

\begin{proof} First, let us observe that $\var$ is ms-full.
	
	(1) $\Leftrightarrow$ (2). Variety $\var$ is discriminator and it contains the only minimal algebra. So, by virtue of Corollary \ref{cor-disc-unique-single-proj}, (1) and (2) are equivalent.
	
	(2) $\Leftrightarrow$ (3) follows from Theorem \ref{th-two}.
	
	(3) $\Leftrightarrow$ (4) follows from the observation that quasiidentity $\rho$ is rejected in an algebra $\Alg{A}$ if and only if $\Alg{A}$ has an element $\alg{a}$ such that $\Box\alg{a} = \Box\neg\alg{a} = \zero$ (note that $\Box\alg{a} = \Box\neg\alg{a}$ is equivalent to $\Box\alg{a} = \Box\neg\alg{a} = \zero$).
\end{proof}

\begin{cor} \label{cor-fpsubalg} Any finitely presented subalgebra of $\Falgv{\omega}$ is projective.
\end{cor}
\begin{proof} Suppose $\Alg{A}$ is a finitely presented subalgebra of $\Falgv{\omega}$. It is clear that $\Alg{A}$ is not a trivial algebra. Also, it is clear that $\two$ is a homomorphic image $\Falgv{\omega}$, and by Proposition \ref{pr-sub-nontr}, $\two$ is a homomorphic image of $\Alg{A}$, and application of Theorem \ref{th-FP-Proj} completes the proof. 
\end{proof}

\begin{cor} \label{cor-fpsubalgalg} A finitely presented algebra $\falg{X,\Delta}{\var}$ is projective if and only if formula $\bigwedge\{\delta \in \Delta\}$ is satisfiable in $\two$. Hence, there is an algorithm deciding by $\Delta$ whether $\falg{X,\Delta}{\var}$ is projective.
\end{cor}

Let us observe that equivalence of properties (2) and (4) of Theorem \ref{th-FP-Proj} means that the set $\var_{mh}$ forms a quasivariety. Hence we have the following.

\begin{cor} \label{cor-wsprim} A subquasivariety $\qvar \subseteq \var$ is primitive if and only if it admits quasiidentity $\rho$.  
\end{cor}

Let us observe that conditions of Theorem \ref{th-FP-Proj} hold for $\VWSFi$, $\VHRI$, $\HDP_n$ and $\DHt_n$, hence Theorem \ref{th-FP-Proj} and all its corollaries hold for these varieties. 

\section*{Acknowledgement} The author is in debt to V.~Marra whose hint allowed to extend the results from $\VWSFi$-algebras to discriminator varieties.

\bibliographystyle{plain}
\def\cprime{$'$}

\end{document}